\documentclass[12pt]{article}

\usepackage{graphicx}%
\usepackage{multirow}%
\usepackage{amsmath,amssymb,amsfonts}%
\usepackage{amsthm}%
\usepackage{mathrsfs}%
\usepackage[title]{appendix}%
\usepackage{xcolor}%
\usepackage{textcomp}%
\usepackage{manyfoot}%
\usepackage{booktabs}%
\usepackage{algorithm}%
\usepackage{algorithmicx}%
\usepackage{algpseudocode}%
\usepackage{listings}%
\usepackage{enumerate}
\usepackage{cite}
\numberwithin{equation}{section}
\newtheorem{theorem}{Theorem}[section]
\newtheorem{proposition}{Proposition}[section]
\newtheorem{lemma}{Lemma}[section]
\newtheorem{corollary}{Corollary}[section]
\usepackage[left=2cm,right=2cm,
    top=2cm,bottom=2cm,bindingoffset=0cm]{geometry}

\begin{document}


\date{}

\title{On commutative invariants for modules over crossed products of minimax nilpotent linear groups}

\author{ Anatolii V. Tushev 
\thanks{The author has received funding through the MSCA4Ukraine project, which is funded by the European Union (ID 1232926)} \\
        Justus Liebig University Giessen\\
        Giessen 35390, Germany;\\
        {\it E-mail: Anatolii.Tushev@math.uni-giessen.de}\\
       }

\maketitle

\begin{abstract}
Let $N$ be a minimax nilpotent torsion-free normal subgroup of a soluble group $G$ of finite rank, $R$ be a finitely generated commutative domain and $R*N$ be a crossed product of $R$ and $N$. In the paper we construct a correspondence between an $R*N$-module $W$ and a finite set $M$ of equivalent 
classes of prime ideals minimal over $Ann_{kA}(W/WI)$, where $kA$ is a group algebra of an abelian minimax group $A$ and $I$ is an appropriative $G$-invariant ideal of $RG$. It is shown  that if $Wg \cong W$ for all $ g \in g $ then the action of the group $G$ by conjugations on $N$ can be extended to an action of the group $G$ on the set $M$. The results allow us to  apply methods of commutative algebra to the study of $W$.

\end{abstract}

\section{Introduction}

A group $G$ is said to have finite (Prufer) rank if there is a positive integer $m$ such that any finitely generated subgroup of $G$ may be generated by $m$ elements; the smallest $m$ with this property is the rank $r(G)$ of $G$. A group $G$ is said to be of finite torsion-free rank if it has a finite series each of whose factor is either infinite cyclic or locally finite; the number $r_{0} (G)$ of infinite cyclic factors in such a series is the torsion-free rank of $G$. If a group $G$ has a finite series each of whose factor is either cyclic or quasi-cyclic then $G$ is said to be minimax. If in such a series all infinite factors are cyclic then the group $G$ is said to be polycyclic. 
\par
Let $G$ be an abelian group and $ t(G) $ be the torsion subgroup of $ G $. Let $p \in \pi(t(G))$  and $G_p$ be the 
Sylow $ p $-subgroup of $t(G)$, where $\pi(t(G))$ is the set of prime divisors of orders of elements of $ t(G) $. 
Then we can define the total rank ${r_t}(G)$ of $G$ by the following formula: ${r_t}(G) = r(G/t(G)) + \sum\nolimits_{p \in \pi(t(G)) } r(G_p) $. A soluble group has finite abelian total rank, or is a soluble $FATR$-group, if it has a finite series in which each factor is abelian of finite total rank. Many results on the construction of soluble $FATR$-groups can be found in \cite{LeRo04}. 
\par

A ring $R * G$ is called a crossed product of a ring $R$ and a group $G$ if $R \le R * G$ and there is an injective mapping ${\varphi ^*}:g \mapsto \bar g$ of the group $G$ to the group of units $U(R * G)$ of  $R * G$ such that each element $a \in R * G$ can be uniquely presented as a finite sum $a = \sum\nolimits_{g \in G} {{a_g}} \bar g$, where ${a_g} \in R$. The addition of two such sums is defined component-wise. The multiplication is defined by the formulas $\bar g\bar h = t(g,h)\overline {gh} $ and $r\bar g = \bar g({\bar g^{ - 1}}r\bar g)$, where $g,h \in G$,$r \in R$, ${\bar g^{ - 1}}r\bar g \in R$ and $t(g,h)$ is a unit of  $R$ (see \cite{Pass89}). This notion of the crossed product was introduced in  \cite{Bovd1963, Bovd1964}
  \par 
If $a = \sum\nolimits_{g \in G} {{a_g}} \bar g \in R * G$, then the set $Supp(a)$ of  elements $g \in G$ such that ${a_g} \ne 0$ is called the support of the element $a$. Let $H$ be a subgroup of  $G$ then the set of  elements $a \in R * G$ such that $Supp(a) \subseteq H$ forms a crossed product $R * H$ contained in $R * G$. If the  subgroup $H$ is normal and $P$ is a $\bar G$-invariant ideal of the ring $R * H$ then it is not difficult to verify that  the quotient ring ${R * G} / {PR * G}$ is a crossed product $({R * H}/ {P}) * (G /H)$ of the quotient ring ${{R * H} /P}$ and the quotient group $G / H$. In particular, if $RG$ is a group ring, $D$ is a normal subgroup of the group  $G$ and $P$ is a $G$-invariant ideal of the group ring $RD$ then the quotient ring ${{RG} / {PRG}}$ is a crossed  product $({RD} /P) * (G / D)$ of the quotient ring ${RD} / P$ and the quotient group $G / D$. 
\par
The last example  shows the main way in which crossed products arise in studying of group rings and modules over them. If $W$ is an $RG$-module that is annihilated by the ideal $P$ then $W$ can be considered as an ${{RG} / {PRG}}$-module, where ${{RG} / {PRG}}$ is a crossed  product $({RD} /P) * (G / H)$ of the quotient ring ${RD} / P$ and the quotient group $G / H$. This situation arises quite often when studying the structure of $RG$-modules, which gives rise to the need to study modules over such crossed products. 
\par
It should be noted that at present there are few methods suitable for studying modules over group rings of soluble groups outside the class of polycyclic groups. One such method is based on the following techniques introduced by Brookes in  \cite{Broo88} for the case of polycyclic groups. Let $N$ be a group and let $K$ be a normal subgroup of $N$ such that the quotient group $N/K$ is torsion-free minimax abelian. Let $R$ be a ring and let $W$ be a finitely generated $RG$-module. Let $I$ be an $N$-invariant ideal of $RK$ such that $\left|H/I^{\dag } \right|<\infty$ and $k=R/(R\bigcap I)$ is a field, where $ I^{\dag }=G \cap(I+1)$. Then the derived subgroup of the quotient group $H/I^{\dag} $ is finite and hence, by \cite[Lemma 2.1(i)]{Tush2022-1}, $H/I^{\dag} $ has a central subgroup  $A$ of finite index. So, the quotient module $(W)_I$ may be considered as a finitely generated $kA$-module. Let $\mu$ be the set of prime ideals of $kA$ minimal over $ann_{kA}(W/WI)$. Below, we use the notation  $(W)_I = W/WI$ . The results of works \cite{Tush2000,Tush2002,Tush2022,Tush2022-1} show that the property of the set $\mu$ can be used in studying the properties of the module $W$ and this gives good results. The main result of the paper, Theorem 3.1 shows that the approach described above can be extended to the case of modules over crossed products.

 \section{ Grand culling ideals and modules over crossed products of torsion-free minimax nilpotent groups}

Let $H$ be a subgroup of a group $G$, the subgroup $H$ is said to be dense in $G$ if for any $g \in G$ there is an integer $n \in \mathbb{N}$ such that ${g^n} \in H$. If ${g^n} \in G\backslash H$ for any $n \in \mathbb{N}$ and any $g \in G\backslash H$ then the subgroup $H$ is said to be isolated in $G$. If the group $G$ is locally nilpotent then the isolator $i{s_G}(H) = \{ g \in G|{g^n} \in H$ for some $n \in \mathbb{N}\} $ of $H$ in $G$ is a subgroup of $G$ and if $H$ is a normal subgroup then so is $i{s_G}(H)$. 
If $G$ is a group then the spectrum $Sp(G)$ of $G$ is the set of prime integers $p$ such that the group $G$ has an infinite $p$-section. 


\begin{proposition}\label{P2.1}
Let $L$ a nilpotent $FART$-group and let $K$ be a normal subgroup of $L$ such that the quotient group $L/K$ is polycyclic. Let $R$ be a finitely generated commutative domain such that $char\, R\notin Sp(L)$ and let $M$ be a faithful $RL$-module. Suppose that the subgroup $K$ contains an isolated abelian $L$-invariant subgroup $D$ such that  $P=Ann_{RD} (M)$ is a prime $L$-invariant faithful ideal of $RD$ such that $P\bigcap R=0$. If the module $M$ is ${RK/ PRK} $-torsion-free then for any nonzero element $0\ne a\in M$ there is a finitely generated subgroup $H\le L$ such that $aRL=aRH\otimes _{RH} RL$.
\end{proposition}
\begin{proof} Let $\tilde{R}={RD/P }  $, $\tilde{L}={L/ D} $ and $\bar{K}={K/D} $. Then ${RL/PRL} =\tilde{R}*\tilde{L}$ is a crossed product of the commutative domain $\tilde{R}$ and the nilpotent group $\tilde{L}$ and ${RK/PRK} =\tilde{R}*\tilde{K}$ is a subring of $\tilde{R}*\tilde{L}$. Since $MP=0$, we can consider $M$ as an $\tilde{R}*\tilde{L}$-module. By \cite[Lemma 4.3]{Tush2025}, there exists a Noetherian partial right ring of quotients $\tilde{R}*\tilde{L}(\tilde{R}*\tilde{K})^{-1}$ and , by \cite[Proposition 4.1(iii)]{Tush2025}, there exists a cyclic $\tilde{R}*\tilde{L}(\tilde{R}*\tilde{K})^{-1} $-module $W=a\tilde{R}*\tilde{L}(\tilde{R}*\tilde{K})^{-1} $ such that $a\tilde{R}*\tilde{L}\leq W$ and $W=\left\{ms^{-1}\, |\, m\in a\tilde{R}*\tilde{L},\, s\in \tilde{R}*\tilde{K}\right\}$. Then the arguments on the proof of \cite[Theorem 4.5]{Tush2025} shows that there is a finitely generated dense subgroup $\tilde{X}$ of $\tilde{R}*\tilde{L}$ such that $a\tilde{R}*\tilde{L}=a\tilde{R}*\tilde{X}\otimes _{\tilde{R}*\tilde{X}} \tilde{R}*\tilde{L}=\oplus _{t\in T} a\tilde{R}*\tilde{X}t$, where $T_{\tilde{X}} $ is a right transversal for $\tilde{X}$ in $\tilde{L}$. Therefore, $aRL=aRX\otimes _{RX} RL=\oplus _{t\in T} aRXt$, where $X$ is the preimage of $\tilde{X}$ in $L$ and $T$ is a right transversal for $X$ in $L$. It follows from \cite[Chap. 2, Lemma 2.1]{Karp} that we can replace $X$ with $L$ and hence we can assume that the quotient group $\tilde{L}={L/ D} $ is finitely generated. Then there is a finitely generated dense subgroup $H$ of $L$ such that $L=DH$. Let $D_{H} =D\bigcap H$ then it is not difficult to show that $D_{H} $ is a normal subgroup of $L$. By \cite[Theorem 3.6(ii)]{Tush2025}, there is a finitely generated subgroup $E$ of $D$ such that $P=(P\bigcap RE)RD$. Evidently, we can choose the subgroup $H$ such that $E\le D_{H}$ and hence $P=P_{H} RD$, where $P_{H} =P\bigcap RD_{H}$. It easily implies that $PRL=P_{H}RL$ and, as the ideal $P$ and the subgroup $D_{H}$ are $L$-invariant, we can conclude that $P_{H}$ is a $L$-invariant ideal of $RD_{H} $. Since $PRL=P_{H} RL$, we have $RL/PRL=RL/P_{H} RL$ and hence $\tilde{R}*\tilde{L}=\tilde{R}_{H} *(\tilde{D}\times \tilde{H})$, where the quotient group $\tilde{D}=D/D_{H} $ is torsion. By \cite[Proposition 4.1(i)]{Tush2025}, there exists a partial right ring of quotients $\tilde{R}*\tilde{L}(\tilde{R}_{H} )^{-1} =\tilde{R}_{H} *(\tilde{D}\times \tilde{H})(\tilde{R}_{H} )^{-1}$, where $\tilde{H}=H/D_{H} $ and $\tilde{D}=D/D_{H} $. Then, as the ideal $P_{H} $ and the subgroup $D_{H} $ are $L$-invariant, we see that $(\tilde{R}(\tilde{R}_{H} )^{-1} )*\tilde{L}=(\tilde{R}_{H} (\tilde{R}_{H} )^{-1} )*(\tilde{D}\times \tilde{H})$. Since the quotient group $\tilde{D}=D/D_{H} $ is torsion, the domain $\tilde{R}(\tilde{R}_{H} )^{-1} $ is algebraic over the subfield $\tilde{R}_{H} (\tilde{R}_{H} )^{-1} $ and hence $\tilde{R}(\tilde{R}_{H} )^{-1} $ is a field. Then it follows from \cite[Proposition 1.6]{Pass89} that the ring $Q=(\tilde{R}(\tilde{R}_{H} )^{-1} )*\tilde{L}=(\tilde{R}_{H} (\tilde{R}_{H} )^{-1} )*(\tilde{D}\times \tilde{H})$ is Noetherian. 
\par
Evidently, $aQ\le W$. Let $I=ann_{Q} (a)$, as the ring $Q$ is Noetherian, the ideal $I$ is finitely generated. It easily implies that we can choose the finitely generated dense subgroup $H\le L$ such that all generators of $I$ are contained in $(\tilde{R}_{H} (\tilde{R}_{H} )^{-1} )*\tilde{H}$ and hence $I=(I\bigcap (\tilde{R}_{H} (\tilde{R}_{H} )^{-1} )*\tilde{H})(\tilde{R}_{H} (\tilde{R}_{H} )^{-1} )*(\tilde{D}\times \tilde{H})$. Then it follows from \cite[Lemma 4.4]{Tush2025} that $a(\tilde{R}_{H} (\tilde{R}_{H} )^{-1} )*(\tilde{D}\times \tilde{H})=\oplus _{t\in \tilde{T}} (a(\tilde{R}_{H} (\tilde{R}_{H} )^{-1} )*\tilde{H}))t$, where $\tilde{T}$ is a right transversal to $H/D_{H} $ in $L/D_{H} $. Therefore, as $aRH\le a(\tilde{R}_{H} (\tilde{R}_{H} )^{-1} )*\tilde{H}$, we can conclude that $aRL=\oplus _{t\in T} (aRH)t$, where $T$ is a right transversal to $H$ in $L$ and hence $aRL=aRH\otimes _{RH} RL$.
\end{proof}

Let $K$ be a normal subgroup of a group $L$ such that $K\le L$ and the quotient group ${ L/K}$ is free abelian with free generators ${ Kx}_{{ 1}} { , Kx}_{{ 2}} { , ..., Kx}_{{ n}} $. We say that $\chi { =\{ <K,\{ x}_{{ j}} { \} }_{{ j}\in { J}} { >|J}\subseteq { \{ 1,..., n\} \} }$ is a full system of subgroups of $L$ over K.
 \par  
Let $R$ be a domain, V be an ${ RK}$-module and let $W$ be an image of ${ V}\otimes _{RK} { RL}$ under an ${ RL}$-module homomorphism $\alpha $. Put $\chi { (W) = \{ X}\in \chi |{ ker }\alpha \cap { (V}\otimes _{RK} { RX) = 0\} }$  and let ${ M}\chi { (W) }$ be the set of maximal elements of $\chi { (W) }$. 

 \begin{lemma}\label{L2.1}
Let $N$ be a minimax nilpotent torsion-free group which has a finite series $D\le K\le L$ of normal subgroups such that the subgroup $D$ and the quotient group $N/K$ are torsion-free abelian, the subgroup $L$ is dense and the subgroup $D$ is isolated in $N$. Suppose that the quotient group $L/K$ is free abelian and let $\chi $ be a full system of subgroups of $L$ over K. Let $R$ be a commutative domain, $0\ne W$ be an $RN$-module, $P$ be an $N$-invariant faithful prime ideal of $RD$ which annihilates $W$ and such that the module $W$ is $RK/PRK$-torsion-free. Then there is an $RN$-submodule $0\ne V\le W$ such that for any element ${ 0 }\ne {  a }\in {  V}$:
\begin{description}
\item[(i)] the module $V$ is $RX/PRX$-torsion-free for any $X\in \chi (aRL)$ and $RX/PRX$-torsion for any $X\in \chi \backslash \chi (aRL)$;
\item[(ii)] $\chi (aRL)=\chi (bRL)$ for any element ${ 0 }\ne {  b }\in {  V}$;
\item[(iii)] in the case where $N=L$, for any subgroup ${ X}\in { M}\chi { (aRL)}$ we have ${ bRL}\cap { aR(X}\bigcap { H)}\ne { 0}$ for any element ${ 0 }\ne {  b }\in {  aRL}$ and any finitely generated dense subgroup ${ H }\le {  L}$. In particular, ${ bRL}\cap { aRH}\ne { 0}$, ${ bRL}\cap { aRX}\ne { 0}$ and the module ${ aRL }$is uniform. 
 \end{description}
 \end{lemma}

\begin{proof} It is easy to note that in the proof the module $W$ may be replaced by any its proper submodule. Since $P$ annihilates $W$, we can consider $W$ as $RN/PRN$ module. In its tern the quotient ring $RN/PRN$ may be considered as a crossed product $\tilde{R}*\tilde{N}$, where $\tilde{R}=RD/P$ and $\tilde{N}=N/D$. So, we can consider $W$ as an $\tilde{R}*\tilde{N}$-module. 
 \par  

(i) Evidently for any $X\in \chi $ the module $W$ is $RX/PRX$-torsion-free  ($RX/PRX$-torsion) if and only if $W$ is $\tilde{R}*\tilde{X}$-torsion-free ($\tilde{R}*\tilde{X}$-torsion), where $\tilde{X}=X/D$.

Let $X\in \chi $ if the module $W$ is $RX/PRX$-torsion-free then $X\in \chi (aRL)$ for any element ${ 0 }\ne {  a }\in {  W}$. If $Ann_{R*X} w\ne 0$ for some $0\ne w\in W$ then we may change   $W$ by $wR*N$. As $X$ is a normal subgroup of ${ N}$, it follows from \cite[Lemma 2.5]{Tush2022-1} that the module $V=wR*N$ is $\tilde{R}*\tilde{X}$-torsion and hence $X\in \chi \backslash \chi (a{ R}L)$ for any element ${0}\ne {a}\in {V}$. Taking $Y\in \chi \backslash \{ X\} $ and repeating the above arguments we obtain a ${ R}N$-submodule $0\ne V_{1} \le V$ which is either $RY$-torsion or ${ R}Y$-torsion-free. Continuing this process, we see that it is terminated because the set $\chi $ is finite.

 \par  

 (ii) The assertion follows from (i).

 \par  

(iii) Let $\tilde{L}=L/D$, by (i), we can assume that for any ${ 0}\ne w\in W$ and ${ X}\in { M}\chi { (w}\tilde{{ R}}{ *}\tilde{{ L}}{ )}$ the ${ R*N}$-module ${ wR*N}$ is $\tilde{R}*\tilde{X}$-torsion-free and hence, by  \cite[Lemma 2 (i)]{Tush2002}, there exists an $\tilde{R}*\tilde{L}{ (}\tilde{R}*\tilde{X}{ )}^{{ -1}} $-module ${ w}\tilde{R}*\tilde{L}{ (}\tilde{R}*\tilde{X}{ )}^{{ -1}} $. Then it easily follows from the maximality of X that ${ w}\tilde{R}*\tilde{L}{ (}\tilde{R}*\tilde{X}{ )}^{{ -1}} $ has finite dimension over the division ring $\tilde{R}*\tilde{X}{ (}\tilde{R}*\tilde{X}{ )}^{{ -1}} $. Therefore, we can choose the element ${ 0}\ne { v}\in wRL$ such that ${ V=v}\tilde{R}*\tilde{L}{ (}\tilde{R}*\tilde{X}{ )}^{{ -1}} $ is a simple $\tilde{R}*\tilde{L}{ (}\tilde{R}*\tilde{X}{ )}^{{ -1}} $-module. Put $V_{1} =vRL=v\tilde{R}*\tilde{L}$. Then for any $0\ne a,b\in V_{1} $ we have ${ V}=a\tilde{R}*\tilde{L}{ (}\tilde{R}*\tilde{X}{ )}^{{ -1}} ={ b}\tilde{R}*\tilde{L}{ (}\tilde{R}*\tilde{X}{ )}^{{ -1}} $ and hence ${ a}\in { b}\tilde{R}*\tilde{L}{ (}\tilde{R}*\tilde{X}{ )}^{{ -1}} $, it easily implies that ${ b}\tilde{{ R}}{ *}\tilde{{ L}}\cap { a}\tilde{{ R}}{ *}\tilde{{ X}}\ne { 0}$. Since ${ b}\tilde{{ R}}{ *}\tilde{{ L}}={ bRL}$ and ${ a}\tilde{{ R}}{ *}\tilde{{ X}}={ aRX}$, we can conclude that ${ bRL}\cap { aRX}\ne { 0}$. 

Since ${ H}\cap { X}$ is a dense subgroup of X, it follows from  \cite[Lemma 1(iii)]{Tush2002} that ${ bRL}\cap { aR(H}\cap { X)}\ne { 0}$. The last equations also shows that ${ bRL}\cap { aRH}\ne { 0}$ and the module $V_{1} =vRL$ (and hence any its proper submodule) is uniform. Thus, we obtained a submodule $0\ne V_{1} =vRL$ such that the assertion (iii) holds for the chosen subgroup ${ X}\in { M}\chi { (vRL)}$ and elements of $V_{1} $. 

 \par  

Suppose that there is a subgroup $Y\in { M}\chi { (vRL)\backslash \{ X\} }$. Applying the above arguments to $Y$ and the submodule $V_{1} $ we may obtain a submodule $0\ne V_{2} \le V_{1} $ such that the assertion (iii) holds for the chosen subgroups ${ X,Y}\in { M}\chi { (vRL)}$ and elements of $V_{2} $. Continuing this process we see that it is terminated because the set $\chi $ is finite. 
 \end{proof}

Let $R$ be a ring, $G$ be a group and let $J$ be a right ideal of the group ring $RG$. The ideal $J$ is said to be faithful if  $J^{\dag } =G\bigcap (1+J)=1$. 
\par  
Let $K$ be a normal subgroup of $G$, $I$ be a $G$-invariant ideal of the group ring $RK$ then $I^{\dag } =K\bigcap (I+1) $ is a $G$-invariant subgroup of $K$. We say that the ideal $I$ is $G$-grand if $R / (R\bigcap I ) =k$ is a field, $\left|K/ I^{\dag }  \right|<\infty $ and $I=(RF\bigcap I )RK$, where $F$ is a $G$-invariant subgroup of $K$ such that $I^{\dag } \le F$ and the quotient group ${F/I^{\dag } } $ is abelian. If $N$ is a subgroup of $G$ such that $K\le N$ and ${F/ I^{\dag } } $ is a central section of $N$ then we say that the ideal $I$ is $N$-central (see \cite{Tush2023,Tush2024}). 
\par
Let $K$ be a normal subgroup of a  group $N$  such that the quotient group $N/K$ is torsion-free abelian of finite rank. Let $I^{\dag } $ be an $N$-invariant subgroup of $K$ such that the quotient group $K/I^{\dag } $ is finite. As the quotient group $K/I^{\dag } $ is finite and the quotient group $N/K$ is torsion-free abelian of finite rank, it follows from \cite[Lemma 2.2(i)]{Tush2022-1} that $N/I^{\dag } $ has a torsion-free abelian characteristic central subgroup $A$ of finite index. For any subgroup $X$ of $N$ such that $K\le X$ we denote  $A_{X} =A\bigcap X/I^{\dag } $.  

 \begin{lemma}\label{L2.2}

Let $N$ be a minimax nilpotent torsion-free group which has a finite series $D\le K\le L$ of normal subgroups such that the subgroup $D$ and the quotient group $N/K$ are torsion free abelian, the subgroup $L$ is dense, the subgroup $D$ is isolated in $N$ and  the quotient group $L/K$ is free abelian. Let $R$ be a finitely generated commutative domain and $P$ be an $N$-invariant  faithful  prime ideal of $RD$ such that $char\, (RD/P)\notin Sp(N)$. Let  $0\ne W$ be an $RN$-module  which is annihilated by the ideal $P$ and which is $RK/PRK$-torsion-free. Let $I$ be an $N$-grand ideal of $RK$ such that $P\le I$, $A\le L/I^{\dag }$ be a central dense free  abelian subgroup of $N/I^{\dag } $ and $k=R/(R \cap I)$. Then there is an $RL$-submodule $0\ne V\leq W$ such that for any element $0\ne a\in V$ and any element $0\ne b\in aRN$: 
\begin{description}
\item[(i)] $bRL$ is not isomorphic to any proper section of $agRL$ for any $g\in N$;
\item[(ii)] the $kA$-module $(bRL)_{I}$ has a finite series each of whose quotient is isomorphic to some section of the $kA$-module $(aRL)_{I}$; 
\item[(iii)] there is an $RL$-submodule $0\ne cRL\le aRL$ such that the $kA$-module $(cRL)_{I}$ has a finite series each of whose quotient is isomorphic to some section of the $kA$-module $(bRL)_{I}$.
 \end{description}
 \end{lemma}
\begin{proof} 
 (i). It follows from Lemma  \ref{L2.1}(i) that for any element ${ 0 }\ne {  a }\in {  W}$ and any subgroup ${ X}\in { M}\chi { (aRL)}$ the module $W$ is $RX/PRX$-torsion-free. According to Lemma  \ref{L2.1}(iii), there exists an $RL$-submodule $0\ne V\le W$ such that for any element $0\ne a\in V$ and any element ${ 0 }\ne {  b }\in {  aRL}$ the following relation holds ${ bRL}\cap { aRX}\ne { 0}$. So, for any submodule ${ 0 }\ne U\le {  aRL}$ the quotient module ${  aRL/U}$ is $RX/PRX$-torsion. Since $K\le X$ and the quotient group $N/K$ is abelian, it is easy to see that $X$ is a normal subgroup of $N$ and therefore the quotient module ${agRL/U}$ is $RX/PRX$-torsion for any ${ RL}$-submodule ${ 0 }\ne U\le { agRL}$ and any element $g\in N$. Therefore we can conclude that any proper section of $agRL$ is $RX/PRX$-torsion. On the other hand, the module $aRN$ is $RX/PRX$-torsion-free and a contradiction is obtained. 
 \par  
(ii) At first, we assume that $0\ne b\in aRL$. By Proposition  \ref{P2.1}, there exists  a finitely generated dense subgroup ${ H }\le {  L}$ such that ${ aRL = aRH}\otimes _{{ RH}} { RL}$. Obviously,   one can  choose the subgroup $H$ such that $0\ne b\in aRH$ and $I^{\dag } H=L$.  Then, evidently, $ bRL = bRH\otimes _{{ RH}} { RL}$. 
 \par  
Let $K_{1} =K\cap H$ and $I_{1} =I\cap RK_{1} $. It follows from  \cite[Lemma 2.2.6]{Tush2000} that $(aRL)_{I}\cong (aRH)_{I_1} $ and $(bRL)_{I}\cong (bRH)_{I_1} $, where $(bRL)_{I}$ and $ (bRH)_{I_1} $ are considered as  $kA$-modules. Thus, it is sufficient to show that $(bRH)_{I_1} $ has a finite series with factors isomorphic to some section of $(aRH)_{I_1} $ considered as  $kA$-modules. 
 \par  
The ideal $I_{1} $ is $H$-grand and hence there exists an $H$ -invariant subgroup $F\le K_{1}$ such that the quotient group $F/ I_{1}^{\dag }$ is abelian and $I_{1} =(RF\bigcap I _{1} )RK_{1} $. As $\left|F/ I_{1}^{\dag }  \right|<\infty $, we have $\left|H:C\right|<\infty $ , where $C=C_{H} (F/ I_{1}^{\dag }  )$. The arguments of the proof of  \cite[Lemma 2.4.1(i)]{Tush2000} shows that $(RF\bigcap I _{1} )RC$ is a polycentral ideal of $RC$. As $\left|H:C\right|<\infty $, we see that $aRH$ is a finitely generated $RC$ -module and, as $I_{1} =(RF\bigcap I _{1} )RK_{1} $, we can conclude that $aRH(RF\bigcap I _{1} )RC=aRHI_{1} $. Then the assertion follows from  \cite[Lemma 2.4.1(ii)]{Tush2000}. 
 \par  
 Suppose now that $0\ne b\in aRN$. We can present the element $b$  in the form $b=b_{1} g_{1} +...+b_{n} g_{n} $ , where $b_{i} \in aRL$ and $g_{i} \in N$. Then one can see that $bRL\cong (\oplus _{i=1}^{n} b_{i} RLg_{i} )/U$, where $U$ is an $RL$-submodule of $\oplus _{i=1}^{n} a_{i} g_{i} RL$, and hence $(bRL)_{I}\cong (\oplus _{i=1}^{n} ((b_{i}RL)_{I})g_{i} )/U_{1} $, where $U_{1}$ is an $RL$-submodule of $\oplus _{i=1}^{n} ((b_{i}RL)_{I})g_{i} $.  $A$ is a central subgroup of  $N/I^{\dag }$ and hence $((b_{i}RL)_{I})g_{i} )\cong (b_{i}RL)_{I}$ considered as $kA$-modules. Therefore, $(bRL)_{I}\cong (\oplus _{i=1}^{n} ((b_{i}RL)_{I}))/U_{1} $, where $U_{1} $ is an $RL$-submodule of $\oplus _{i=1}^{n} ((b_{i}RL)_{I})$. As it was shown above, the quotient module $(b_{i}RL)_{I}$ has a finite series with factors isomorphic to some sections of $(aRL)_{I}$.
 \par  
(iii). The element $b$ may be presented in the form $b=b_{1} g_{1} +...+b_{n} g_{n} $ , where $b_{i} \in aRL$ and $g_{i} \in N$. We  prove by induction on $n$  that $bRL$ contains a submodule $dRL$ isomorphic to $cgRL=cRLg$, where $0\ne c\in aRL$ and $g\in N$.
 Obviously, $bRL+(b_{2} g_{2} +...b_{n} g_{n} )RL=b_{1} g_{1} RL+(b_{2} g_{2} +...b_{n} g_{n} )RL$ and hence $bRL+(b_{2} g_{2} +...b_{n} g_{n} )RL/(b_{2} g_{2} +...b_{n} g_{n} )RL=b_{1} g_{1} RL+(b_{2} g_{2} +...b_{n} g_{n} )RL/(b_{2} g_{2} +...b_{n} g_{n} )RL$. If $bRL\bigcap (b_{2} g_{2} +...b_{n} g_{n} )RL \neq 0$ then $bRL\cong b_{1} g_{1} RL/(b_{1} g_{1} RL\bigcap (b_{2} g_{2} +...b_{n} g_{n} )RL)$ but it contradicts (i). Thus,  $b_{1} g_{1} RL\bigcap (b_{2} g_{2} +...b_{n} g_{n} )RL=0$. Therefore, $bRL\cong b_{1} g_{1} RL$  and hence we can put $c=b_{1} $ and $g=g_{1}$. 
 \par  
If there is $0\ne d_{1} \in bRL\bigcap (b_{2} g_{2} +...b_{n} g_{n} )RL$ then $d_{1} $ can be presented in the form $d_{1} =e_{2} g_{2} +...e_{n} g_{n} RL$, where $e_{i} \in aRL$. As $0\ne d_{1} RL\le bRL$, changing the element $b$ by $d_{1} $ we can apply the induction hypothesis. 
 \par   
 By (ii), $(dRL)_{I}$ has a finite series each of whose quotient is isomorphic to some section of $(bRL)_{I}$ considered as $kA$-modules and it follows from step 1 that $(cRLg)_I$ has such a series. Since $A$ is a central subgroup of $N/I^{\dag } $, we can conclude that $(cRLg)_I$ and $(cRL)_{I}$ are isomorphic as $kA$-modules. 
 \end{proof}

Let $R$ be a ring, let $V$ be an $R$-module and $U$ be a submodule  of  $V$. An ideal $I$ of $R$ culls $U$ in $V$ if $VI<V$ and  $VI+U=V$ (see  \cite{Broo88}). 
\par
 Let $R$ be a commutative domain, $K$ be a normal subgroup of a group $H$. According to  \cite[Introduction, p. 89]{Broo88}, we say that an $RK$-module $0\ne V$ is $H$-ideal critical if for any submodule $0\ne V_{1} \le V$ there is an $H$-invariant ideal $I$ of $RK$ such that: 
 \begin{enumerate}[(i)]
\item[(i)] $I$ culls $V_{1} $ in $V$;
\item[(ii)] $I$ has the weak Artin-Rees property that is, if $U$ is any finitely generated $RK$-module with submodule $U_{1} $ then $UI^{n} \bigcap U_{1} \le U_{1} I$ for some $n\in {\mathbb N}$. 
 \end{enumerate}
 
 If $G$ is a group then the $FC$-center $\Delta(G)=\left\{{g\in G|\left|{G:{C_G}(g)}\right|<\infty}\right\}$ of $G$ is a characteristic subgroup of $G$. If $N$ is a normal subgroup of $G$ then $\Delta_N(G)=\Delta(G) \bigcap N $. If the normal subgroup $N$ is torsion-free minimax nilpotent then, by \cite[Lemma 1(ii)]{Tush2024}, $\Delta_N(G)$ is a central $G$-invariant isolated subgroup of $N$.


 \begin{proposition}\label{P2.2}

Let $N$ be a minimax nilpotent torsion-free normal subgroup of a soluble group $G$ of finite torsion-free rank, let $D=\Delta_N (G) = N \bigcap \Delta (G)$, $K$ be a $G$-invariant subgroup of $N$  such that $D\leq H$ and $H$ be a finitely generated subgroup of $K$. Let $R$ be a finitely generated commutative domain and let $P$ be a $G$-invariant faithful prime ideal of $RD$ such that $char \, RD/P \notin Sp(N)$. Let  $J$ be a right ideal of $RH$ such that $P \cup RH < J$.  Then:   
 \begin{description}
 \item[(i)] there is a $G$-grand ideal $I$ of $RK$ such that $P \leq I$   and  ${ RH }\cap {I}$ culls $J$ in  $RH$; 
\item[(ii)] if $H=N$ then $RH/PRH $  is critical  uniform $RH$-module. 
 \end{description}
 \end{proposition}

\begin{proof} (i) Let $a \in J\setminus PRH$ then, by \cite[Theorem ]{Tush2024} there is a $G$-grand ideal $I$  of $RK$ such that the image $\tilde{a}$ of $a$ in 
$RK/I$ is invertible in $RK/I$. Therefore, $\tilde{a}$ is not a zero divisor in $RK/I$ and hence $\tilde{a}$ is not a zero divisor in $(RH+I)/I$. Since $(RH+I)/I \cong RH/(RH \bigcap I)$, it implies that  $\tilde{a}$ is not a zero divisor in  $RH/(RH \bigcap I)$. Then, as $RH/(RH  \bigcap I)$ is a finite dimensional $k$-algebra, where  $k=R\bigcap I$, $\tilde{a}$ is invertible in $RH/(RH \ bigcap I)$. Therefore, $aRH+I=RH$ and hence, as 
$aRH \leq J $, the assertion follows. 
 \par  
(ii). By  \cite[Lemma 2.4.1(i)]{Tush2000}, any $H$-grand  ideal    $I$ of  $RK$  is polycentral and it follows from  \cite[Theorem 6.12]{Wehr09} that $I$ has the weak Artin-Rees property. Then (i) implies that $RH/PRH$ is $RH$-ideal critical. Evidently, $RK/PRK $ is a crossed product $\tilde{R}*\tilde{H}$  of the commutative domain  $\tilde{R}=RD/P$ and the torsion-free nilpotent group $\tilde{H}= H/D $. Then It follows from  \cite[Corollary 37.11]{Pass89} that $RK/PRK$ is an Ore domain and hence $RK/PRK$ is uniform. 
 \end{proof}

\begin{lemma}\label{L2.3}
Let $H$ be a torsion-free finitely generated nilpotent group which has a series $D \leq K$ of normal subgroups such that  the subgroup $D$ is isolated abelian and the quotient group $H/K$ is free abelian.   Let $\chi $ a full system of subgroups over $K$. 
Let $R$ be a finitely generated domain and let $0\ne W$ be an $RH$-module which is annihilated by an $H$-invariant faithful prime ideal $P$ of  $RD$ and such that the module $W$ is $RK/PRK$-torsion-free. Then there are a  cyclic $RH$-submodule ${ 0 }\ne {a}RH\le {W}$ and a right ideal $J$ of $RK$ such that if $I$ is an $H$-grand ideal of $RK$ such that $P\le I$ and $I$ culls $J$ in $RK$ then  for any cyclic $RH$-submodule ${0}\ne bRH\le aRH$ we have $X\in \chi (bRH)$ if and only if  the quotient module $(bRH)_I$ is not $kA_{X}$-torsion, where $A$ is a characteristic central torsion-free subgroup of finite index in $H/I^{\dag }$, $A_{X} =A\bigcap X/I^{\dag } $ and $k=R/(R \cap I)$. 
 
 \end{lemma}

\begin{proof}
{\bf  Step 1.} It easily follows from  \cite[Lemma 8]{Broo88} that there exists a right ideal $J_{0} $ of  {$RK$} such that $PRK\le J_{0} $ and if an $H$-grand ideal $I$ of $RK$ culls $J_{0} $ in $RK$ (i.e. $I$ culls $aJ_{0} $ in $V=aRK\cong RK/PRK$) then $X\in \chi (aRH)$ if   only if $(aRH)_{I}\otimes _{kA_{X} } Q_{X} \ne 0$, where $Q_{X} $ denotes the field of fractions of the domain $kA_{X} $. We should note that in fact the relation $(aRH)_{I}\otimes _{kA_{X} } Q_{X} \ne 0$ means that the quotient module $(aRH)_{I}$ is not $kA_{X} $-torsion. 
 \par  
{\bf Step 2.} Suppose that $X\in \chi (aRH)$ then the arguments of the proof of  \cite[Lemma 14]{Broo88} show that there exists a right ideal ${ J}_{{ X}} \le { RK}$ such that $PRK\le J_{X} $ and if an $H$-grand ideal $I$ of $RK$ culls $J_{X} $ in $RK$ then $aRXI=aRHI\cap aRX$. By  \cite[Lemma 3.4.]{Tush2022-1}, $I^{n} $ also culls $J_{X} $ and hence  we can conclude that 
   $aRXI^{n} =aRHI^{n} \cap aRX$.                          
\par 
    Let $0 \neq c \in aRH$. It easily follows from the definition of ${ M} \chi (aRH)$ that ${ aRX}\cong { RX/PRX}$. Then it  follows from Proposition \ref{P2.2}(ii) that $\bigcap _{{ n}\in { N}} { aRXI}^{{ n}} { =0}$. 
        Then there is $m\in {\mathbb N}$ such that $c\in aRXI^{m-1} \backslash aRXI^{m} $. Therefore, $ cRX/(cRX\cap {aRXI}^{m})\cong  (cRX+aRXI^{ m})/aRXI^{ m} $ is a non-zero submodule of ${ aRXI}^{{ m-1}} { /aRXI}^{{ m}} $. 
 \par  
     Let ${ U =(RX/PRX)I}^{{ m-1}} { /(RX/PRX)I}^{{ m}} $. Evidently, $${ U = ((RK/PRK)I}^{{ m-1}} { /(RK/PRK)I}^{{ m}} { )}\otimes _{k(K/I^{\dag })} { (K/I^{\dag })}.$$ Then, as ${ A}\cap  (K/I^{\dag }) = 1$, we see that $U$ is ${ kA}_{{ X}} $-torsion-free and hence so is the quotient module $aRXI^{m-1} /aRXI^{m}$.  It easily implies that the quotient module  ${ (cRX}+aRX{ I}^{{ m}} { )/}aRX{ I}^{{ m}} \cong { cRX/(cRX}\cap aRX{ I}^{{ m}}  )$ is also ${ kA}_{{ X}} $-torsion-free. 
  \par  
     Since $c\in aRX$ and  $aRXI^{n} =aRHI^{n} \cap aRX$  we can conclude that 
   $aRXI^{n} \bigcap cRX=aRHI^{n} \bigcap cRX$
    and hence $ cRX/(cRX \cap aRHI^{ m}  )$ is a $kA_X $-torsion-free module. Since $cRX\le aRXI^{m-1}$, we can see that $ cRH \le aRHI^{ m-1} $ and it implies  $cRHI\le aRHI^{ m}$. Therefore, ${ cRX}\cap cRHI \leq cRX\cap aRHI^m$. So , as ${ cRX/(cRX}\cap  aRHI^m)$ is a $kA_X$-torsion-free module, the quotient module $cRX/(cRX \cap  cRHI) \simeq (cRX+cRH )/cRHI\leq   (cRH)_I$ is not $kA_X$-torsion. 
 \par  
{\bf Step 3.} By Lemma \ref{L2.1} we can choose an element $0\ne a\in W$ such that for any $0\ne b\in aRH$ we have       $\chi (aRH)=\chi (bRH) $  and  $ aRX\cap bRH\ne 0$ for any $X\in { M} \chi (aRH)$.     
 \par  
  Let $J_{0} $ be a right ideal of $RK$ from Step 1 defined for the element $a$ and let $J_{X} $ be a right ideal of $RK$ from Step 2   defined for the element $a$ and for a subgroup $X\in \chi (aRH)$. Put $J=J_{0} \bigcap (\bigcap _{X\in \chi (aRH)} J_{X} )$ and let $I$ be an $H$-grand ideal of $RK$ such that $P \leq I$ and which culls $J$ in $RK$.  Then it follows from  \cite[Lemma 6]{Broo88}] that the ideal $I$ culls $J_{0} $ and $J_{X} $ in $RK$ for each $X\in \chi (aRH)$. 
 \par  
 Let ${0}\ne {bRH}\le {a}RH$ and suppose that for some $X\in \chi $ the quotient module $(bRH)_I$ is not $kA_{X} $-torsion. Then it follows from Lemma \ref{L2.2}(ii) that the quotient module $ (aRH)_I$ is not $kA_{X}$-torsion. Since the ideal $I$ culls $J_{0} $ in $RK$, Step 1 shows that $X\in \chi (aRH)$. Then, as 
 $   \chi (aRH)=\chi (bRH)$, we can conclude that  $X\in \chi (bRH)$. 
 \par  
 Let now $X\in \chi (bRH)$ and show that the quotient module 
 $(bRH)_I$ is not $kA_{X} $-torsion. Evidently, $X\le Y$ for some $Y\in M\chi (bRH)$ and if we show that $(bRH)_I$ is not $kA_{Y} $-torsion then $(bRH)_I$ is not $kA_{X} $-torsion because $A_{X} \le A_{Y} $. Thus, we can assume that $X\in M\chi (bRH)$ and hence, as $aRX\cap bRH\ne 0$, there is $0\ne c\in aRX\bigcap bRH$. As $0\ne c\in aRX$ and $I$ culls $J_X$ in $RK$, Step 2 shows that the quotient module $(cRH)_I$ is not $kA_X $-torsion. Then, as $0\ne cRH\le bRH$, it follows from Lemma \ref{L2.2}(ii) that the quotient module $(bRH)_I$ is not 
 $kA_{X} $-torsion.  
 \end{proof}

 \begin{proposition}\label{P2.3}
Let $N$ be a minimax nilpotent torsion-free group which has a finite series $D\le K\le L$ of normal subgroups such that the subgroup $D$ and the quotient group $N/K$ are torsion free abelian, the subgroup $L$ is dense, the subgroup $D$ is isolated central in $N$ and  the quotient group $L/K$ is free abelian. Let $R$ be a finitely generated domain and  $P$ be an $N$-invariant prime faithful ideal of $RD$ such that $char\, (RD/P)\notin Sp(N)$. Let $0\ne W$ be an $RN$-module annihilated by $P$ and which is $RK/PRK$-torsion-free.  
Then there exist an element ${ 0 }\ne {  a}\in { W}$, a finitely generated dense subgroup $H$ of $K$ and a right ideal $J$ of $R(H\bigcap K)$  such that if $I$ is an $N$-grand ideal of $RK$ such that $P \leq I $ and $I\bigcap R(K\bigcap H)$ culls $J$ in $R(H\bigcap K)$ then  for any ${ 0 }\ne {  b}\in { aRN}$ we have $X\in \chi (bRL)$ if and only if  the quotient module $(bRL)_{I}$ is not $kA_{X} $-torsion, where $A_{X} =A\bigcap X/I^{\dag } $ and $A$  is a characteristic central torsion-free subgroup of finite index in $N/I^{\dag } $. 
 \end{proposition}
\begin{proof}  Lemmas \ref{L2.1} and \ref{L2.2} allow us to choose a cyclic submodule ${ 0 }\ne {  aRN}\le {  W}$ which satisfies the conditions (i)-(iii) of Lemma \ref{L2.1} and the conditions (ii),(iii) of Lemma \ref{L2.2}  which hold for any cyclic submodule ${ 0 }\ne {  bRN}\le {  aRN}$. 
 \par  
{\bf Step 1.} At first, we suppose that $N=L$. Then Proposition \ref{P2.1} shows that there is a finitely generated dense subgroup $H$ of $L$ such that ${ aRL = aRH}\otimes _{{ RH}} { RL}$. 
 \par  
 By \cite[Lemma 3.1]{Tush2022-1} we may chose the subgroup $H$ such that for any $L$ invariant subgroup $I^{\dag } \le K$ of finite index in $K$ we have $ I^{\dag } H=L$ and
                  $ \chi  =\{ I^{\dag }  X\, \mid \,|X\in \chi _{{ H}}  \}$,                            
  where $$\chi _{H} { =\{ <K}_{{ 1}} { ,\{ x}_{{ j}} { \} }_{{ j}\in { J}} { >|J}\subseteq { \{ 1,..., n\} \} }$$ is a full system of subgroups of $H$ over $K_{1} =K\bigcap H$.    
 \par  
 Let $X\in \chi _{H} (aRH)$ if $I^{\dag } { X}\in \chi \backslash \chi (aRL)$ then it follows from Lemma \ref{L2.1}(i) that the module $aRL$ is $R(I^{\dag } { X)}/PR(I^{\dag } { X)}$-torsion and it follows from  \cite[Lemma 2.2.3(ii)]{Tush2000} that the module $aRL$ is $R{ X}/(P\bigcap R(D\bigcap H))R{ X}$-torsion, because $X$ is a dense subgroup of $I^{\dag } { X}$. Therefore, $X\notin \chi _{H} (aRH)$ and a contradiction is obtained. 
 Suppose now that $I^{\dag } { X}\in \chi (aRL)$, where $X\in \chi _{H} $ then it follows from Lemma \ref{L2.1}(ii) that  the module $aRL$ is $R(I^{\dag } { X)}/PR(I^{\dag } { X)}$-torsion-free. Therefore, $aRH$ is $R{ X}/(P\bigcap R(D\bigcap H))R{ X}$-torsion-free and hence $X\in \chi _{H} (aRH)$. Thus, we have the following  relation:  
   \begin{equation}\label{2.1}
                      \chi (aRL)=\{ I^{\dag } { X}\, \mid \, |X\in \chi _{H} (aRH)\}.                
 \end{equation}
 \par
 According to Lemma \ref{L2.3}(iii), we can choose the element $a$ such that there exists a right ideal $J\geq P$ of $R(K \cap H )$  such that if $I_{1} $ is an $H$-grand ideal of $R(K \cap H )$ which culls $J$ in $R(K \cap H )$ then for any cyclic $RH$-submodule ${ 0 }\ne {b}RH\le {a}RH$ we have $X\in \chi _{H} (bRH)$ if and only if the quotient module $(bRH)_{I_1} $ is not $kA_{X} $-torsion, where $A_{X} =A\bigcap X/I_{1}^{\dag } $ and $A$ is a central subgroup of finite index in $H/I_{1}^{\dag } $. 
 \par  
 Suppose that there exists  an $L$-grand ideal $I$ of $RK$ such that $I_{1} =I\bigcap R(K \cap H )$ culls $J$ in $R(K \cap H ) $. Since  $I^{\dag } H=L$, it follows from   \cite[Lemma 2.2.6]{Tush2000}   that $(aRL)_{I}$ and $(aRH)_{I_1} $ are isomorphic as $kA$-modules, where $A$ is a central subgroup of finite index in $L/I^{\dag } $. Then it follows from Lemma \ref{L2.3}(iii) and   (\ref{2.1}) that $X\in \chi (aRL)$ if and only if the quotient module $(aRL)_{I}$ is not $kA_{X}$-torsion, where $A_{X} =A\bigcap X/I^{\dag } $ and $A$  is a central subgroup of finite index in $L/I^{\dag } $. 
 \par  
  Let $0\ne b\in aRL$, it follows from Lemma \ref{L2.1}(iii) that   there is an element $0\ne c\in bRL\bigcap aRH$. Therefore,  we have a chain of submodules $cRH\otimes _{RH} RL=cRL\le bRL\le aRL=aRH\otimes _{RH} RL$. It follows from Lemma \ref{L2.1}(ii) that  $\chi (cRL)=\chi (bRL)=\chi (aRL)$. Since $cRL=cRH\otimes _{RH} RL$, the arguments applied in the case $aRL$ show that $X\in \chi (cRL)=\chi (bRL)=\chi (aRL)$  if and only if $(cRL)_{I}$ is not $kA_{X} $-torsion. Let $X\in \chi (cRL)=\chi (bRL)=\chi (aRL)$ then $(cRL)_{I}$ is not $kA_{X} $-torsion.  Therefore, it follows from Lemma \ref{L2.2} that $(bRL)_{I}$ is not $kA_{X} $-torsion because $cRL\le bRL$. If $(bRL)_{I}$ is not $kA_{X} $-torsion then it follows from Lemma \ref{L2.2} that $(aRL)_{I}$ is not $kA_{X} $-torsion because $bRL\le aRL$ and it means that  $X\in \chi (cRL)=\chi (bRL)=\chi (aRL)$.
 \par  
{\bf Step 2.} By step 1, the element $0\ne a\in W$ may be chosen  such that the assertion of lemma holds for any element $0\ne b\in aRL$. Besides, $aRN$ satisfies the conditions (i)-(iii) of Lemma \ref{L2.1} and the conditions (ii), (iii) of Lemma \ref{L2.2} and these conditions hold for any cyclic submodule ${ 0 }\ne {  bRN}\le {  aRN}$. 
 \par  
Let $0\ne b\in aRN$ and let $X\in \chi (bRL)$. Suppose that the quotient module $(bRL)_{I}$ is $kA_{X} $-torsion. By Lemma \ref{L2.2}(iii), there exists an $RL$-submodule $0\ne cRL\le aRL$ such that $(cRL)_{I}$ has a finite series each of whose quotient is isomorphic to some section of $(bRL)_{I}$ considered as $kA$-modules. Then the quotient module $(cRL)_{I}$ is $kA_{X} $-torsion. By Lemma \ref{L2.1}(ii), $X\in \chi (cRL)$ and by, step 1, the quotient module $(cRL)_{I}$ is not $kA_{X} $-torsion. So, we have a contradiction and hence the quotient module $(bRL)_{I}$ is not $kA_{X} $-torsion. 
 \par  
 Suppose now that the quotient module $(bRL)_{I}$ is not $kA_{X} $-torsion. By Lemma \ref{L2.2}(ii), $(bRL)_{I}$ has a finite series each of whose quotient is isomorphic to some section of $(aRL)_{I}$ considered as $kA$-modules and hence the quotient module $(aRL)_{I}$ is not $kA_{X} $-torsion. Then, by step 1, $X\in \chi (aRL)$ and, by Lemma \ref{L2.1}(ii),  $X\in \chi (bRL)$
 \end{proof}


\section{A set of commutative invariants for modules over crossed products of torsion-free minimax nilpotent groups}

Let $S$ be a commutative Noetherian ring and let $I$ be an ideal of $S$. Let ${\mu _S}(I)$ be the set of prime ideals of $S$ minimal over $I$, by \cite[Chap. II, \S 4, Corollary 3]{Bour}, the set ${\mu_S}(I)$ is finite. 
\par  
Let $P$ be a prime ideal of  $S$ and let ${S_P}$ be the localization  of $S$ at the ideal $P$. Let $M$ be an $S$-module, the support $Sup{p_S}M$ of the module $M$ consists of all prime ideals $P$ of $S$ such that ${M_P} = M{ \otimes _S}{S_P} \ne 0$ (see \cite[Chap. II, \S 4]{Bour}). By \cite[Chap. IV, \S 1, Theorem 2]{Bour}, if $S$ and $M$ are Noetherian then the set ${\mu_S}(M)$  of minimal elements of $Sup{p_S}M$ coincides with ${\mu_S}(Ann_S(M))$, where ${\mu_S}(Ann_S (M))$ is the set of prime ideals of $S$ which are minimal over $An{n_S}M$. Thus, we have
\begin{equation}
{\mu_S}(M) = {\mu_S}(Ann_{S}(M)).      \label{3.1} 
\end{equation}

	Let $K$ be a normal subgroup of a group $G$. Let $L$ be a   subgroup of $G$ such that $K \leqslant L \leqslant G$ and the   quotient group $L/K$ is finitely generated free abelian. Let $R$ be
	 a commutative domain, $I$ be a $G$-grand ideal of the group ring $RK$ and $k=R/(R \cap I)$. Then the quotient group $K/I^\dag$ is finite and, as the quotient group $L/K$ is finitely generated free abelian, it follows from \cite[Lemma 2.2(i)]{Tush2022-1} that
	 $L/I^\dag$  has a characteristic central torsion-free subgroup $A$ of finite index. Since the quotient group $L/K$ is finitely generated free abelian, the subgroup $A$ is finitely generated free abelian. By \cite[Proposition 1.6]{Pass89}, the group ring $kA$ is a Noetherian. Let $W$ be a finitely generated $RL$-module then $(W)_I$ is a finitely generated $kA$ -module and hence $(W)_I$ is a Noetherian $kA$-module. So, the finite set ${\mu _{kA}}((W)_I) = {\mu _{kA}}(An{n_{kA}}((W)_I))$ is defined. 
\par


\begin{lemma}\label{L3.1}
Let $L$ be a minimax torsion-free nilpotent group which has a 
series $D \leq K$ of  normal subgroups such that the subgroup $D$ is abelian, the quotient group $L/D$ is torsion-free and the quotient group $L/K$ is free abelian. Let $T$ be a subgroup of finite index in  $L$ such that $K \leqslant T$. Let $R$ be a finitely generated commutative domain and $P$ be an $L$-invariant faithful prime ideal of the group ring $RD$ such that $char\, (RD/P)\notin Sp(N)$.  Let $W$ be a cyclic $RL$-module which is annihilated by the ideal  $P$ and which is $RK/PRK$-torsion-free. Let $I$ be an $L$-grand ideal of $RK$ such that $P\leq I$ and $W \ne WI$. Let $A$ be a central torsion-free subgroup of finite index in $L/{I^\dag }$ , $B$ be a subgroup of finite index in $A \cap (T/{I^\dag })$ and $k= R/(R \cap I)$.  Then:
 \begin{description}
\item[(i)] $Sup{p_{kB}}  {(W_1)_I}  \subseteq Sup{p_{kB}}  (W)_I$  for any cyclic $RT$-submodule  ${W_1} \leqslant W$; 
\item[(ii)] there exists a cyclic $RL$-submodule $0 \ne V \leqslant W$ such that  ${\mu _{kB}}((V_1)_I) = {\mu _{kB}}((V)_I)$  for any cyclic $RL$-submodule  $0 \ne {V_1} \leqslant V$ and any  subgroup $B\leq A$ of finite index in $A$. 
\end{description}
\end{lemma}

     \begin{proof} (i). Since $\left| {L:T} \right| < \infty $, we see that $aRL$ is a finitely generated $RT$-module. If the group $L$ is finitely generated, the assertion is proved in \cite[Lemma 5(i)]{Tush2002}. 
\par  
     Consider now the general case. Suppose that $W = aRL$ and ${W_1} = bRT$, where $b \in aRL$. By Proposition \ref{P2.1}, there is a finitely generated subgroup $H \leqslant L$ such that $W = aRH{ \otimes _{RH}}RL$. Evidently, taking the subgroup $H$ bigger if it is necessary, we can assume that $b \in aRH$ and $L = {I^\dag }(K \cap H)$. Therefore, ${W_1} = bR(H \cap T){ \otimes _{R(H \cap T)}}RT$ and, as $L = {I^\dag }(K \cap H)$, we have $L = {I^\dag }H$ and $T = {I^\dag }(H \cap T)$. Then it follows from   \cite[Lemma 2.2.6]{Tush2000}   that there are $RH$-module isomorphism $(aRL)_{I} \simeq (aRH)_{I' }$ and $R(H \cap T)$-module isomorphism  $(bRT)_I \simeq (bR(H \cap T))_{I'}$, where $I' = RH \cap I$ and these isomorphisms induce $kB$-module isomorphisms. So, the assertion follows from the considered above case, where the group $L$ is finitely generated. 
\par  
(ii) To prove the assertion we can repeat the arguments of the proof of  \cite[Lemma 4.2]{Tush2022-1} applying (i) instead of  \cite[Lemma 4.2(i)]{Tush2022-1}.
\end{proof}


\begin{lemma}\label{L3.2}
Let $N$ be a minimax torsion-free nilpotent group which has a 
series $D \leq K\leq L$ of  normal subgroups such that the subgroup $D$ is isolated abelian, the quotient group $N/K$ is torsion-free abelian and  the  the quotient group $L/K$ is free abelian. Let $R$ be a finitely generated commutative domain and $P$ be an $N$-invariant faithful prime ideal of the group ring $RD$ such that $char\, (RD/P)\notin Sp(N)$.  Let $W$ be an $RN$-module which is annihilated by the ideal  $P$ and which is $RK/PRK$-torsion-free. Let $I$ be an $N$-grand ideal of $RK$ such that $P\leq I$ and $W \ne WI$. Let $A$ be a central torsion-free subgroup of finite index in $L/{I^\dag }$ and $k= R/(R \cap I)$. Then  there exists a cyclic $RN$-submodule $0 \ne aRN \leqslant W$ such that ${\mu _{kB}}((bRL)_{I}) = {\mu _{kB}}((aRL)_{I})$ for any element $0 \ne b \in aRN$ and any  subgroup $B\leq A$ of finite index in $A$.  
\end{lemma}

    \begin{proof} It follows from \cite[Lemma 4.1(iv)]{Tush2022-1}that it is sufficient to show that  there exists a cyclic $RN$-submodule $0 \ne U \leqslant W$ such that ${\mu _{kA}}((bRL)_{I}) $ for any element $0 \ne b \in U$.
\par      
      By Lemma \ref{L2.2}(ii, iii), there exists a cyclic $RL$-submodule $0 \ne V \leqslant W$ such that for any  element $0 \ne a \in V$ and any element $0 \ne b \in aRN$ the $kA$-module $(bRL)_{I}$ has a finite series each of whose quotient is isomorphic to some section of the $kA$-module $(aRL)_{I}$. Moreover, by Lemma \ref{L2.2}(iii), there exists an $RL$-submodule $0 \ne cRL \leqslant aRL$ such that the $kA$-module $(cRL)_{I}$ has a finite series each of whose quotient is isomorphic to some section of the $kA$-module $(bRL)_{I}$. Then it follows from \cite[Ch. II, \S 4, Proposition 16]{Bour} that 
\begin{equation}
         Sup{p_{kA}}(cRL)_{I} \subseteq  Supp_{kA}(bRL)_{I} \subseteq Sup{p_{kA}}(aRL)_{I}.      \label{3.0}
\end{equation}

Accordin to Lemma \ref{L3.1}(ii), we also can choose the element $a \in V$ such that ${\mu _{kA}}((V_1)_I) = {\mu _{kA}}((aRL)_{I})$  for any cyclic $RL$-submodule  $0 \ne {V_1} \leqslant aRL$ and hence we can assume that
\begin{equation}
          {\mu _{kA}}((cRL)_{I}) = {\mu _{kA}}((aRL)_{I})      \label{3.2}
\end{equation}
If $P \in {\mu _{kA}}(aRL)_I$  then it follows from (\ref{3.2}) and (\ref{3.0})  that $P \in Supp_{kA}(bRL)_{I}$ and the second embedding of (\ref{3.0}) shows that $P \in {\mu _{kA}}(bRL)_{I}$. Thus, ${\mu _{kA}}(aRL)_I \subseteq {\mu _{kA}}((bRL)_{I})$. Then it follows from  the second embedding of (\ref{3.0})
that ${\mu _{kA}}(aRL)_I = {\mu _{kA}}((bRL)_{I})$.
\end{proof}

	 Let $A$ be a torsion-free abelian group of finite rank and let $k$ be a field. If $P$ and $Q$ are ideals of $kA$ then we write $P \approx Q$ if  $P \cap kB = Q \cap kB$ for some finitely generated dense subgroup $B \leqslant A$. Then $ \approx $ is an equivalence relation on the set of all prime ideals of $kA$ and we denote by $\left[P \right]$ the class of equivalence containing an ideal $P$. If a group $G$ acts on $A$ then we obtain an action of $G$ on the set of equivalent prime ideals of $kA$ which is given by ${\left[ P \right]^g} = \left[ {{P^g}} \right]$. 
\par  
	 If $B$ is a dense subgroup of $A$ and $P$ is a prime ideal of $kB$ then, as $kA$ is an integer domain over $kB$, it follows from \cite[Chap. V, \S 2, Theorem 1]{Bour} that there is a prime ideal $Q$ of $kA$ such  that $Q \cap kB = P$ and we put  ${[P]_{kA}} = [Q]$. If $\mu $ is a set of prime ideals of $kB$ then we put ${[\mu ]_{kA}} = \{ {[P]_{kA}}|P \in \mu \} $. 
\par
	Let $G$ be a group and let $K$ be a normal subgroup of $G$. Let $N$ be a nilpotent subgroup of $G$ such that $K \leqslant N \leqslant G$ and the quotient group $N/K$ is torsion-free abelian of finite rank. Let $R$ be a commutative domain and let $I$ be a $G$-grand ideal of $RK$. Then, as the quotient group $K/I^\dag$ is finite and the quotient group $N/K$ is torsion-free abelian of finite rank, it follows from \cite[Lemma 2.2(i)]{Tush2022-1} that $N/I^\dag$ has a characteristic central torsion-free subgroup $A$ of finite index. 
\par  
Let $L$ be a dense subgroups of $N$ such that $K \leqslant L$ and the quotient groups $L/K$ is finitely generated. Then $A \cap L/I^\dag$ is a dense central finitely generated torsion-free subgroup of $A$. 
\par  
Let $W$ be an $RN$-module  and $aRL$ be a cyclic $RL$-module generated by an element $0 \ne a \in W$. By \cite[Lemma 4.4(i)]{Tush2022-1}, there is a finitely generated dense subgroup $A_L \leqslant A \cap L/I^\dag$ 
such that for any subgroup $X$ of finite index in ${A_L}$ the mapping ${\mu _{k{A_L}}}(An{n_{k{A_L}}}((aRL)_{I})) \to {\mu _{kX}}(An{n_{kX}}((aRL)_{I}))$ given by $P \mapsto P \cap kX$ is bijective , where ${\mu _{k{A_L}}}(An{n_{k{A_L}}}((aRL)_{I}))$ is the set of minimal primes over $An{n_{k{A_L}}}((aRL)_{I})$ and ${\mu _{kX}}(An{n_{kX}}((aRL)_{I}))$ is the set of minimal primes over $An{n_{kX}}((aRL)_{I})$. 
\par  
As ${A_L}$ is a finitely generated  subgroup of finite index in $A \cap L / I^\dag$, we can conclude that ${A_L}$ is a dense finitely generated  subgroup of $A$ and $(aRL)_{I}$ is a finitely generated $k{A_L}$-module. Then it follows from \cite[Proposition 1.6]{Pass89}  that the domain $k{A_L}$ is Noetherian and $(aRL)_{I}$ is a Noetherian $k{A_L}$-module. Thus, the sets $Sup{p_{k{A_L}}}((aRL)_{I})$ and ${\mu _{k{A_L}}}((aRL)_{I})$ are well defined. Then, according to (\ref{3.1}), we have  
\begin{equation}\label{3.4}
                         {\mu _{k{A_L}}}((aRL)_{I}) = {\mu _{k{A_L}}}(An{n_{k{A_L}}} ((aRL)_{I}))                  
\end{equation}
Thus, the set ${\mu _{k{A_L}}}((aRL)_{I})$ is defined for any $0 \ne a \in W$ and we put ${[{\mu _{k{A_L}}}((aRL)_{I})]_{kA}} = \left\{ {{{[P]}_{kA}}|P \in {\mu _{k{A_L}}}((aRL)_{I})} \right\}$. Then, by (\ref{3.4}),   
\begin{equation}\label{3.5}
[\mu _{kA_L}((aRL)_{I})]_{kA} = [\mu _{kA_L}(Ann_{kA_L}((aRL)_{I}))]_{kA}.
\end{equation}
 So, according to \cite[Lemma 4.4(ii),(iii)]{Tush2022-1}, the set 
$${[{\mu _{k{A_L}}}((aRL)_{I})]_{kA}} = \left\{ {{{[P]}_{kA}}\, |\,  P \in {\mu _{k{A_L}}}(An{n_{k{A_L}}}({(aRL)_{I}}))} \right\}$$ 
is finite and does not depend on the choice of the subgroup ${A_L}$ which meets the conditions of \cite[Lemma 4.4(i)]{Tush2022-1}. 
Everywhere below in the definition of the set ${[{\mu _{k{A_L}}}((aRL)_{I})]_{kA}}$ we assume that the subgroup ${A_L}$ meets the conditions of \cite[Lemma 4.4(i)]{Tush2022-1}.


\begin{proposition}\label{P3.1}
Let $N$ be a minimax torsion-free nilpotent group which has a 
series $D \leq K$ of  normal subgroups such that the subgroup $D$ is isolated abelian and the quotient group $N/K$ is torsion-free abelian. Let $R$ be a finitely generated commutative domain and $P$ be an $N$-invariant faithful prime ideal of the group ring $RD$ such that $char\, (RD/P)\notin Sp(N)$.  Let $W$ be an $RN$-module which is annihilated by the ideal  $P$ and which is $RK/PRK$-torsion-free. Let $I$ be an $N$-grand ideal of $RK$ such that $P\leq I$ and $W \ne WI$. Let $A$ be a central torsion-free subgroup of finite index in $L/{I^\dag }$ and $k= R/(R \cap I)$. Then there exists a cyclic $RN$-submodule $0 \ne V \leqslant W$ such that ${[{\mu _{k{A_L}}}((aRL)_{I})]_{kA}} = {[{\mu _{k{A_M}}}((bRM)_I)]_{kA}}$ for any elements $0 \ne a,b \in V$ and any dense subgroups $L,M \leqslant N$ such that $K \leqslant L \cap M$ and the quotient groups $L/K$ and $M/K$ are finitely generated. 
\end{proposition}

\begin{proof}  Let $L$ be a dense subgroup of $N$ such that  $K \leqslant L$ and the quotient group $L/K$ is finitely generated. According to Lemma \ref{L3.2}  we can choose an element $0 \ne a \in W$ such that ${\mu _{k{B_L}}}((aRL)_{I}) = {\mu _{k{B_L}}}((bRL)_{I})$ for any element $0 \ne b \in aRN$  and any subgroup $B_L$ of finite  index in $A_L$.  
\par  
 Let $M$ be a dense subgroup of $N$ such that $K \leqslant M$ and the quotient group $M/K$ is finitely generated. It is easy to note that $T = L \cap M$ is a subgroup of finite index in $L$ and $M$ . It easily implies that  we can choose a finitely generated dense subgroup  ${A_T} \leqslant A \cap (T/{I^\dag })$ such that${A_T} \leqslant {A_L} \cap {A_M}$. Then it is sufficient to show that ${[{\mu _{k{A_L}}}((bRL)_{I})]_{kA}} = {[{\mu _{k{A_T}}}((bRT)_{I})]_{kA}} = {[{\mu _{k{A_M}}}((bRM)_{I})]_{kA}}$. It is sufficient to prove only the identity 
 ${[{\mu _{k{A_L}}}((bRL)_{I})]_{kA}} = {[{\mu _{k{A_T}}}((bRT)_{I})]_{kA}}$ because the identity ${[{\mu _{k{A_T}}}((bRT)_{I})]_{kA}} = {[{\mu _{k{A_M}}}((bRM)_{I})]_{kA}}$ is analogous. It follows from \cite[Lemma 4.4(iii)]{Tush2022-1} that ${[{\mu _{k{A_L}}}((bRL)_{I})]_{kA}} = {[{\mu _{k{A_T}}}((bRL)_{I})]_{kA}}$  and hence it is sufficient to show that  
 ${[{\mu _{k{A_T}}}((bRT)_{I})]_{kA}} = {[{\mu _{k{A_T}}}((bRL)_{I})]_{kA}}$. 
\par
Since $T \leqslant L$, we see that $bRL = \sum\limits_{i = 1}^n {bRT{g_i}} $ for some ${g_i} \in L$. Then   
$(bRL)_{I} \cong (\sum\limits_{i = 1}^n {(bRT{g_i})_I} )/X,$  where $X$ is a $k{A_T}$-submodule of $\sum\limits_{i = 1}^n (bRT{g_i})_I$. As ${A_T}$ is a central subgroup of $N/{I^\dag }$, we can conclude that $(bRT{g_i})_I \cong (bRT)_{I}$ for all ${g_i} \in L$ and hence 
$(bRL)_{I} \cong (\sum\limits_{i = 1}^n ((bRT)_{I})_i)/X.$ 
Then it follows from  \cite[Ch. II, \S 4, Proposition 16]{Bour} that 
$ Sup{p_{k{A_T}}}(bRL)_{I} \subseteq Sup{p_{k{A_T}}}(bRT)_{I}. $
\par 
On the other hand, as  $T \leqslant L$, we have $bRT \leqslant bRL$ and it  follows from Lemma \ref{L3.1}(i) that 
$Sup{p_{k{A_T}}}(bRT)_{I} \subseteq Sup{p_{k{A_T}}}(bRL)_{I}.$  
Thus, we can conclude that  $Sup{p_{k{A_T}}}(bRL)_{I} = Sup{p_{k{A_T}}}(bRT)_{I}$ and hence ${\mu _{k{A_T}}}((bRL)_{I}) = {\mu _{k{A_T}}}((bRT)_{I})$.  Therefore,  ${[{\mu _{k{A_T}}}((bRT)_{I})]_{kA}} = {[{\mu _{k{A_T}}}((bRL)_{I})]_{kA}}$. 
\end{proof}


\begin{corollary}\label{C3.1}
Let $N$ be a minimax torsion-free nilpotent group which has a 
series $D \leq K$ of  normal subgroups such that the subgroup $D$ is isolated abelian and the quotient group $N/K$ is torsion-free abelian. Let $R$ be a finitely generated commutative domain and $P$ be an $N$-invariant faithful prime ideal of the group ring $RD$ such that $char\, (RD/P)\notin Sp(N)$.  Let $W$ be an $RN$-module which is annihilated by the ideal  $P$ and which is $RK/PRK$-torsion-free. Let $I$ be an $N$-grand ideal of $RK$ such that $P\leq I$ and $W \ne WI$. Let $A$ be a central torsion-free subgroup of finite index in $L/{I^\dag }$ and $k= R/(R \cap I)$. Then there exists a cyclic $RN$-submodule $0 \ne V \leqslant W$ which defines a finite set ${M_{kA}}((V)_I) = {[{\mu _{k{A_L}}}((aRL)_{I})]_{kA}}$ of equivalent classes of prime ideals of $kA$ which depends only on the ideal $I$ and the subgroup $A$. 
\end{corollary}

\begin{proof} 
The finiteness of  ${M_{kA}}((V)_I) = {[{\mu _{k{A_L}}}((aRL)_{I})]_{kA}}$ follows from \cite[Lemma 4.4(ii),(iii)]{Tush2022-1} and (\ref{3.5}). By Proposition \ref{P3.1}, ${M_{kA}}((V)_I) = {[{\mu _{k{A_L}}}((aRL)_{I})]_{kA}}$ does not depend on the choice of the subgroup $L$ and the element $0 \ne a \in V$. By \cite[Lemma 4.4(iii)]{Tush2022-1}, 
${ M_{kA}}((V)_I) = {[{\mu _{k{A_L}}}((aRL)_{I})]_{kA}}$ does not depend on the choice of the subgroup ${A_L}$ which meets the conditions of \cite[Lemma 4.4(i)]{Tush2022-1}. 
\end{proof}

 Propositions \ref{P2.2}, \ref{P2.3} and Corollary \ref{C3.1} allow us to obtain the following theorem


\begin{theorem}\label{T3.1} 
Let $G$ be a soluble group of finite torsion-free rank ${r_0}(G) < \infty $ which has a series $D \leq K \leq N$ of normal subgroups such that the subgroup $N$ is torsion-free minimax, the quotient group $N/K$ is torsion-free abelian and $D=\Delta_G(N)$. Let $L$ be a dense subgroup of $N$ such that $K \leq L$ and the quotient group $L/K$ is free abelian and let $\chi $ be a full system of subgroups of $L$ over $K$. Let $R$ be a finitely 
generated commutative domain and $P$ be a $G$-invariant faithful prime ideal of $RD$ such that $char\, RD/P \notin Sp(N)$. Let $W$ be an $RN$-module which is annihilated by $P$ and  $RK/PRK$-torsion-free. Then there are a cyclic $RN$-submodule $0 \ne V \leqslant W$, a $G$-grand ideal $I$ of $RK$ and a central $G$-invariant subgroup $A$ of finite index in $N/{I^\dag }$ such that 
 $P\leq I$, $k=R/(R\bigcap I)$  and for any element $0 \ne b \in V$ we have:     
 \begin{description}
\item[(i)]  $X \in \chi (bRL)$ if and only if  the quotient module $(bRL)_{I}$ is not $k{A_X}$-torsion, where ${A_X} = A \cap X/{I^\dag }$;    
\item[(ii)]  for any dense subgroup $M \leqslant N$ such that $K \leqslant M$ and the quotient group $M/K$ is finitely generated, the finite set ${M_{kA}}((V)_I) = {[{\mu _{k{A_M}}}((bRM)_I)]_{kA}}$ of equivalent classes of prime ideals of $kA$ depends only on the ideal $I$ and the subgroup $A$;
 \item[(iii)] ${M_{kA}}((Vg)_I) = { ( M_{kA}}{((V)_I))^g} = \{ {[P]^g} = [{P^g}]|[P] \in {M_{kA}}((V)_I)\}$ for any  $g \in G$. 
 \item[(iv)] if $Vg \cong V$ for any  $g \in G$, then $ ( M_{kA}((V)_I))^G = M_{kA}{((V)_I)}$, i.e. $G$ acts on $M_{kA}{((V)_I)}$ by conjugations. 
\end{description}
\end{theorem}

	\begin{proof}  (i) According to  Proposition \ref{P2.3}, there exist a cyclic $RN$-submodule $0 \ne U \leqslant W$, a finitely generated dense subgroup $H$ of $L$ and a right ideal $J$ of $R(H \cap K)$  such that if $I$ is an $N$-grand ideal of $RK$ such that $I \cap R(K \cap H)$ culls $J$ in $R(H \cap K)$ then  for any $0 \ne b \in U$ we have $X \in \chi (bRL)$ if and only if  the quotient module $(bRL)_{I}$ is not $k{A_X}$-torsion, where ${A_X} = A \cap X/{I^\dag }$ and $A$  is central subgroup of finite index in $N/{I^\dag }$. By Proposition \ref{P2.2}, there is a $G$-grand ideal $I$ of $RK$ such that $P\leq I$ and $I \cap R(K \cap H)$ culls $J$ in $R(H \cap K)$. As $N \leqslant G$, we can conclude that the ideal $I$ is an $N$-grand and hence the assertion (i) holds for the ideal $I$ and the submodule $U$. Since $N$, ${I^\dag }$ and $K$ are normal subgroups of $G$ and, by  \cite[Lemma 2.2(i)]{Tush2022-1}, any   abelian-by-finite nilpotent group of finite rank has a characteristic central subgroup of finite index, we can chose the subgroup $A$ such that $A$ is $G$-invariant. Then replacing $W$ by $U$ we can assume that (i) holds for any $0 \ne b \in W$. 
\par  
(ii) According to Corollary \ref{C3.1}, we can choose a cyclic submodule $0 \ne V \leqslant W$ such that for any $0 \ne b \in V$ and any dense subgroup $M \leqslant N$ with $K \leqslant M$ and finitely generated quotient group $M/K$ the finite set ${ M_{kA}}((V)_I) = {[{\mu _{k{A_M}}}((bRM)_I)]_{kA}}$ depends only on the ideal $I$ and the subgroup $A$. 
\par  
(iii)  It  follows from (ii) that  
$ M_{kA}((Vg)_I) =[{\mu _{k{A^g}_M}}((bRM)_{I} g)]_{kA}.$  
because the subgroup $A$ and the ideal $I$ are $G$-invariant.  
Then by (\ref{3.5}), $$M_{kA}((Vg)_I)=[\mu _{kA^g_M}((bRM)_{I}g)]_{kA} = {[\mu _{k{A^g}_M}(Ann_{{kA^g}_M}(bRM)_{I} g)]_{kA}}$$     
and hence $M_{kA}((Vg)_I)=[\mu_{kA^g_M}(Ann_{kA^g_M}(bRM)_I g)]_{kA}=$ ${[{\mu_{kA_M}}({(Ann_{k{A_M}}(bRM)_I)^g})]_{kA}}$. 
 \par 
So, we can conclude that $$M_{kA}((Vg)_I)=[{({\mu _{k{A_M}}}{(Ann_{kA_M}(bRM)_I)^g}]_{kA}}=\{ {[{P^g}]_{kA}}|P \in ({\mu _{k{A_M}}}(An{n_{k{A_M}}}(bRM)_I)\} =$$ $$ \{ {[P]^g}_{kA}|P \in ({\mu _{k{A_M}}}(Ann_{k{A_M}}(bRM)_I)\}={\{ [P]^g = [P^g ]|[P] \in M_{kA}((V)_I)\}= M_{kA}}{((V)_I)^g}$$ 
\par 
(iv) The assertion easily follows from (iii). 
\end{proof}


 \end{document}